\documentclass[12pt]{elsart}
\usepackage{graphics}
\usepackage{graphicx}
\usepackage[dvips]{epsfig}
\usepackage{subfigure}
\usepackage{color}
\usepackage{amssymb,amsmath,amscd}
\usepackage{amsfonts}

\newcommand{\R}{\mathbb{R}}

\newenvironment{proof}{\noindent{\it Proof}\rm.}{\hfill $\Box$}
\newenvironment{proofof}[1]{\bigskip\noindent{\it Proof of~#1.}\rm}{\hfill $\Box$}

\theoremstyle{plain}
\newtheorem{theorem}{Theorem}[section]
\newtheorem{lemma}{Lemma} [section]
\newtheorem{corollary}[theorem]{Corollary}
\newtheorem{proposition}[theorem]{Proposition}

\theoremstyle{definition}

\newtheorem{remark}[theorem]{Remark}
\numberwithin{equation}{section}

\begin{document}
\begin{frontmatter}

\title{Eigenvalue inequalities for Klein-Gordon Operators}
\author{Evans M. Harrell II}
\ead{harrell@math.gatech.edu}
\address{School of Mathematics, Georgia Institute of
Technology, Atlanta, GA 30332-0160 U.S.A.}
\author{Selma Y{\i}ld{\i}r{\i}m Yolcu} \ead{selma@math.gatech.edu}
\address{School of Mathematics, Georgia Institute of
Technology, Atlanta, GA 30332-0160 U.S.A.}

\begin{abstract}

We consider the pseudodifferential operators $H_{m,\Omega}$
associated by the prescriptions of quantum mechanics to the
Klein-Gordon Hamiltonian $\sqrt{|{\bf P}|^2+m^2}$ when restricted to
a compact domain $\Omega$ in ${\mathbb R}^d$. When the mass $m$ is
$0$ the operator $H_{0,\Omega}$ coincides with the generator of the
Cauchy stochastic process with a killing condition on $\partial
\Omega$.  (The operator $H_{0,\Omega}$ is sometimes called the {\it
fractional Laplacian} with power $\frac{1}{2}$, cf. \cite{Gie}.) We
prove several universal inequalities for the eigenvalues $0 <
\beta_1 < \beta_2 \le \cdots$ of $H_{m,\Omega}$ and their means
$\overline{\beta_k} := \frac{1}{k} \sum_{\ell=1}^k{\beta_\ell}$.

Among the inequalities proved are:
\begin{equation}\nonumber
{\overline{\beta_k}} \ge {\rm cst.} \left(\frac{k}{|\Omega|}\right)^{1/d}
\end{equation}
for an explicit, optimal ``semiclassical'' constant, and, for any dimension $d \ge 2$
and any $k$:
\begin{equation}\nonumber
\beta_{k+1} \le  \frac{d+1}{d-1} \overline{\beta_k}.
\end{equation}
Furthermore, when $d \ge 2$ and $k \ge 2j$,
\begin{equation}\nonumber
\frac{\overline{\beta}_{k}}{\overline{\beta}_{j}} \leq
\frac{d}{2^{1/d}(d-1)}\left(\frac{k}{j}\right)^{\frac{1}{d}}.
\end{equation}

Finally, we present some analogous estimates allowing for an external potential energy field,
i.e, $H_{m,\Omega}+ V(\bf x)$, for
$V(\bf x)$ in certain function classes.

\end{abstract}

\end{frontmatter}

\section{Introduction}\label{intro}

The quantum-mechanical operator corresponding to the Klein-Gordon
Hamiltonian is a first-order pseudodifferential operator used to
model relativistic particles in quantum mechanics. On unrestricted
space the part representing kinetic energy $\sqrt{|{\bf P}|^2+m^2}$
can be defined as the square root of $-\Delta + m^2$, where $m$ is a
nonnegative constant corresponding to the mass, in units where the
speed of light is set to $1$. We restrict it to a compact domain in
${\mathbb R}^d$ and designate the quantum version of $\sqrt{|{\bf
P}|^2+m^2}\Big|_{\Omega}$ as $H_{m,\Omega}$. (A full definition of
$H_{m,\Omega}$ is provided below.) The operator $H_{m,\Omega}$ is
positive definite with compact inverse and hence it has purely
discrete spectrum consisting of positive eigenvalues $0 < \beta_1 <
\beta_2 \le \dots$. When $m=0$ the operator $H_{0,\Omega}$ reduces
to the generator of the Cauchy stochastic process
\cite{Sato,BanKul}, and because
\begin{equation}\label{H0vsHm}
H_{0,\Omega} \le H_{m,\Omega} \le  H_{0,\Omega} + m,
\end{equation}
we shall sometimes be able to restrict to this case without of generality.

Our aim is to find analogues for $H_{m,\Omega}$ of some familiar
inequalities of a general nature that apply to the eigenvalues $0 <
\lambda_1 < \lambda_2 \le \cdots$ of the Dirichlet problem for the
Laplacian on a bounded domain $\Omega \in \R^d$. In some of these
the spectrum is constrained by the shape and size of $\Omega$; for
example the volume of $\Omega$ appears in both the Faber-Krahn lower
bound for $\lambda_1$ and in the Weyl estimate of $\lambda_k $ as $k
\rightarrow \infty$.  In addition, there are {\it universal bounds},
whereby  either $\lambda_k$ individually, or else some quantity
involving many eigenvalues such as an average, a gap, or a ratio, is
controlled by a different spectral quantity, independently of the
geometry of $\Omega$.  Various aspects of the well-developed
subjects of geometric and universal bounds are treated,  for
instance, in \cite{Ash,Ban,Bera,Cha,Henrot}.  One way to generate
geometric and universal bounds for the Laplacian is based on
identities for traces of commutators of operators
\cite{Har93,HarMic,HarStu,LevPar,AshHer}, and with the benefit of
hindsight these algebraic methods can be perceived implicitly in
most of the classic universal spectral bounds for Laplacians
\cite{PPW,HilPro,Yan}.  Moreover, comparable universal bounds have
been obtained with the same strategy for  Schr\"odinger operators on
Euclidean spaces \cite{HarStu}, and both Laplacians and
Schr\"odinger operators on embedded manifolds
\cite{Li,YanYau,HarMic,Mic,CY1,CY2,EHI,Har93,Har07}. In many cases
examples can be identified in which the inequalities are saturated.

The plan of attack is to use trace identities to derive universal
spectral bounds and geometric spectral bounds for $H_{m,\Omega}$.
The generator of the Cauchy process, corresponding to the case
$m=0$, is often referred to as the fractional Laplacian and
designated $\sqrt{-\Delta}$.  The latter is, unfortunately,
ambiguous notation, since this operator is distinct from the
operator $\sqrt{-\Delta_\Omega}$ as defined by the functional
calculus for the Dirichlet Laplacian $-\Delta_\Omega$, except when
$\Omega$ is all of ${\mathbb R}^d$.  For this reason we shall avoid
the ambiguous notation when speaking of compact $\Omega$. (For the
spectral theorem and the functional calculus, see, e.g.,
\cite{ReeSim}.)
Whereas several universal eigenvalue bounds, mostly of unknown or
indifferent sharpness, have been obtained for higher-order partial
differential operators such as the bilaplacian (e.g.,
\cite{LevPro,HarMic,CY3,WanXia,WuCao}), and for some first-order
Dirac  operators \cite{Che}, universal bounds for pseudodifferential
operators appear not to have been studied before.

In a final section we study interacting Klein-Gordon operators of the form
\begin{equation}\label{Ham}
H = H_{m,\Omega} + V({\bf x}),
\end{equation}
allowing an external force field.
An additional contemporary motivation for \eqref{Ham}
comes from nanophysics, because when a nonrelativistic
particle travels in a two-dimensional hexagonal structure like
carbon graphene, the effective Hamiltonian operator is relativistic in form,
albeit with a characteristic speed smaller than the speed of light \cite{Wal}.

Klein-Gordon operators can be conveniently defined using the Fourier
transform on the dense subspace of test functions
$C_c^{\infty}(\R^d)$. With the normalization
$$
\widehat{\varphi}({\bf \xi}) = {\mathcal F}\left[\varphi\right] :=
\frac{1}{(2 \pi)^{d/2}} \int_{\R^d}{\exp{(-i {\bf \xi} \cdot {\bf x})}
\varphi({\bf x}) d{\bf x}},
$$
the Laplacian is given by $- \Delta \varphi :=  {\mathcal F}^{-1}
|{\bf \xi}|^2 \widehat{\varphi}({\bf \xi})$, and therefore
\begin{equation} \label{sqrt}
\sqrt{- \Delta + m^2} \varphi :=  {\mathcal F}^{-1} \sqrt{|{\bf \xi}|^2 + m^2} \widehat{\varphi}({\bf \xi}).
\end{equation}

The semigroup generated on $L^2(\R^d)$ is known explicitly, so that,
for instance with $m=0$,

\begin{equation} \label{sgpRd}
\exp{(- \sqrt{-\Delta} t)}\left[\varphi\right]({\bf x}) = p_0(t, \cdot) * \varphi,
\end{equation}

where for $t > 0$  the transition density (= convolution kernel) is
\begin{equation} \label{td}
p_0(t, {\bf x}) := \frac{c_d t}{(t^2+|{\bf x}|^2)^{\frac{d+1}{2}}},
\end{equation}
with
$\displaystyle{c_d} := \frac{d!}{(4 \pi)^{d/2} \Gamma(1+d/2)}$.
(Cf. \cite{BanKul}. We note that $c_d$ is the same ``semiclassical''
constant that appears in the Weyl estimate for the eigenvalues of
the Laplacian.  It is given in \cite{BanKul} and some other sources
as $\pi^{-\frac{d+1}{2}} \Gamma\left(\frac{d+1}{2}\right)$, which is
equal to $c_d$ by an application of the duplication formula of the
gamma function.)

If $\Omega$ is a non-empty, bounded, open subset of $\R^d$, then we define
$H_{m,\Omega}$ as follows. Consider the quadratic form on
$C_c^{\infty}(\Omega)$ given by
$$\varphi \rightarrow \int_{\Omega}{\overline{\varphi}  \sqrt{- \Delta + m^2}\;\varphi}$$
(Here $ \sqrt{- \Delta + m^2}$ is calculated for $\R^d$.) Since this
quadratic form is positive and defined on a dense set, it extends to
a unique minimal positive operator (the Friedrichs extension) on
$L^2(\Omega)$, which we designate $H_{m,\Omega}$.
The semigroup
$\displaystyle{e^{-t H_{m,\Omega}}}$
has an integral kernel $p_{m,\Omega}(t,{\bf x},{\bf y})$,
the form of which is typically not known explicitly.

We remark that the Fourier transform can be more directly applied to
$H_{m,\Omega}$ than to the square root of the Dirichlet Laplacian
according to the functional calculus, which dominates it in the
following sense:

Suppose that
$\varphi \in C^{\infty}_{c}(\Omega)\subset
C^{\infty}_{c}({\mathbb R}^d)$. Then

\begin{eqnarray*}
\langle \varphi, H_{m,\Omega}^2\varphi\rangle=\|H_{m,\Omega}\varphi
\|^2&=&\int_{\Omega} \left|{\mathcal
F}^{-1}\left(\sqrt{|{\bf \xi}|^2+m^2}\hat{\varphi}\right)\right|^2\\
&=&\int_{{\mathbb R}^d} \left|\chi_{\Omega}{\mathcal
F}^{-1}\left(\sqrt{|{\bf \xi}|^2+m^2}\hat{\varphi}\right)\right|^2\\
&\leq&\int_{{\mathbb R}^d} \left|{\mathcal
F}^{-1}\left(\sqrt{|{\bf \xi}|^2+m^2}\hat{\varphi}\right)\right|^2\\
&=&\int_{{\mathbb R}^d} {\overline{\varphi}(-\Delta +m^2)\varphi}\\
&=&\int_{\Omega} {\overline{\varphi}(-\Delta +m^2)\varphi},\\
\end{eqnarray*}

because ${\rm supp}(\varphi) \in \Omega$ and $-\Delta$ is a local
operator. Therefore, if $\beta_k$ denotes the $k^{th}$ eigenvalue of
$H_{m,\Omega}$, and $\lambda_k$ is the $k^{th}$ eigenvalue of
$-\Delta$,

\begin{equation}\label{ibbl}\beta_k\leq \sqrt{\lambda_k + m^2}.\end{equation}

\section{Trace formulae and inequalities for spectra of $H_{m,\Omega}$}

In \cite{HarHer} universal bounds for spectra of Laplacians were
found as consequences of differential inequalities for Riesz means
defined on the sequence of eigenvalues. The strategy here is the
same, as adapted to the
eigenvalues $\beta_j,j = 1, \dots$ of the
first-order pseudodifferential operator
$H_{m,\Omega}$.
However, as the earlier article made heavy use of the fact that the Laplacian
is of second order and acts locally, neither of which circumstance
applies here, the results we
obtain here and the details of the argument are quite different.

An essential lemma is an adaptation of a result
of \cite{HarStu,HarStu2}.

\begin{lemma}\rm{(Harrell-Stubbe)} Let $H$ be a self-adjoint
operator on $L^2(\Omega)$, $\Omega \in \R^d$, with discrete spectrum
$\beta_1 \le \beta_2 \le \dots$.  Denoting the corresponding
normalized eigenfunctions $\{u_j\}$, assume that for a Cartesian
coordinate $x_{\alpha}$, the functions $x_{\alpha} u_j$ and
$x_{\alpha}^2 u_j$ are in the domain of  definition of $H$.  Then
\begin{equation}  \label{R1mod}
\sum_{j: \beta_j \le z}(z-\beta_j)\langle
u_j,\left[x_{\alpha},\left[H,x_{\alpha}\right]\right]u_j\rangle-2\|\left[H,x_{\alpha}\right]u_j\|^2
\le 0,
\end{equation}
and
\begin{equation}  \label{R2mod}
\sum_{j: \beta_j \le z}(z-\beta_j)^2\langle
u_j,\left[x_{\alpha},\left[H,x_{\alpha}\right]\right]u_j\rangle-2(z-\beta_j)\|\left[H,x_{\alpha}\right]u_j\|^2
\le 0.
\end{equation}

\end{lemma}
So that this article is self-contained, we provide a proof of the
lemma.\\

\begin{proof}
Elementary calculations show that, subject to the domain assumptions
made in the statement of the theorem,
$$
\left[H,x_{\alpha}\right]u_j = \left(H - \beta_j\right)x_{\alpha}u_j,
$$
and
$$
\langle{u_j,\left[x_{\alpha},\left[H,x_{\alpha}\right]\right]u_j}\rangle
= 2 \langle{x_{\alpha}u_j,\left(H - \beta_j\right)x_{\alpha}u_j}\rangle.
$$
These two identities can be combined and slightly rearranged to yield:
$$
(z-\beta_j)\langle
u_j,\left[x_{\alpha},\left[H,x_{\alpha}\right]\right]u_j\rangle-2\|\left[H,x_{\alpha}\right]u_j\|^2
$$
$$
\quad\quad\quad\quad\quad\quad
= 2 \langle{\left((z-\beta_j) - (H - \beta_j) \right) x_{\alpha}u_j, (H - \beta_j) x_{\alpha}u_j}\rangle
$$
\begin{equation}\label{HSId}
\quad
= 2 \langle{\left(z - H\right) x_{\alpha}u_j, (H - \beta_j) x_{\alpha}u_j}\rangle.
\end{equation}
Using the completeness of the eigenfunctions of $H$,
$$
(H - \beta_j) x_{\alpha}u_j = \sum_k{(\beta_k - \beta_j) \langle{x_{\alpha}u_j, u_k}\rangle u_k},
$$
so the right side of \eqref{HSId} can be rewritten as
$$
2 \sum_k{\left(z - \beta_k \right)  \langle{u_k, x_\alpha u_j}\rangle}
 \left(\beta_k - \beta_j\right) \langle{x_\alpha u_j, u_k}\rangle = 2 \sum_k{\left(z - \beta_k \right) \left(\beta_k - \beta_j\right)
|\langle{u_k, x_\alpha u_j}\rangle}|^2
$$
\begin{equation}\label{rtside}
\quad\quad\quad\quad\le 2 \sum_{k:\beta_k < z} {\left(z - \beta_k \right) \left(\beta_k
- \beta_j\right) |\langle{u_k, x_\alpha u_j}\rangle}|^2,
\end{equation}
provided that $\beta_j \le z$.  If we now sum \eqref{HSId} over $j$
with $\beta_j \le z$, i.e., the same values of $j$ as for $k$
in \eqref{rtside},
then
after symmetrizing in $j,k,$
$$
  \sum_{j: \beta_j \le
z}(z-\beta_j)\langle
u_j,\left[x_{\alpha},\left[H,x_{\alpha}\right]\right]u_j\rangle-2\|\left[H,x_{\alpha}\right]u_j\|^2
\quad\quad\quad \quad\quad\quad
$$
$$
\quad\quad\quad \quad\quad\quad \le
\sum_{j,k:\beta_k, \beta_j < z} {\left(\left(z -  \beta_k\right) - \left(z -  \beta_j\right) \right)  \left(\beta_k -
\beta_j\right) |\langle{u_k, x_\alpha u_j}\rangle|^2},
$$
which simplifies to
$$
-\sum_{j,k:\beta_k, \beta_j < z} { \left(\beta_k -
\beta_j\right)^2 |\langle{u_k, x_\alpha u_j}\rangle}|^2 \le 0,
$$
as claimed in \eqref{R1mod}.
In order to establish \eqref{R2mod}, multiply \eqref{rtside} by $(z-\beta_j)$ and then
sum on $j$ for $\beta_j < z$.  The summand on the right side is odd in the exchange of
$j$ and $k$, and thus the right side equates to $0$.
\end{proof}

Some consequences of more general forms of the lemma are worked out
in \cite{HarStu2}. Before deriving a differential inequality that
will be useful to control the spectrum, we first follow the strategy
of \cite{HarStu} to obtain a universal bound on $\beta_{n+1}$ in
terms of the statistical distribution of the lower eigenvalues.  For
this purpose we introduce notation for the normalized moments of the
eigenvalues:

{\bf Definition}. For a real number $r$ and an integer $k > 0$,
$\;\overline{\beta_k^r} := \frac{1}{k} \sum_{j=1}^k{\beta_j^r}$.
When $r=1$ we simply write $\overline{\beta_k} =
\overline{\beta_k^1}$.

\begin{theorem}\label{upperbd} If $d \ge 2$, then for each $k$,
the eigenvalues $\beta_k$ of $H_{m,\Omega}$ satisfy
\begin{equation}\label{ugly}
\beta_{k+1} \le \frac{1}{(d-1) \overline{\beta_k^{-1}}} \left(d +
\sqrt{d^2 - (d^2 - 1)\overline{\beta_k} \, \overline{\beta_k^{-1}}
}\right).
\end{equation}
\end{theorem}
Before giving the proof we note two slightly weaker but more
appealing variants of \eqref{ugly} using the
Cauchy-Schwarz inequality, $1 \le \overline{\beta_k} \,
\overline{\beta_k^{-1}}$, with the aid of which the universal bound
simplifies to

\begin{equation}\label{simpleupper}
\beta_{k+1} \le \frac{d+1}{(d-1) \overline{\beta_k^{-1}}} \le
\frac{d+1}{d-1} \overline{\beta_k}.
\end{equation}

In particular,
\begin{equation}\label{funrat}
\frac{\beta_2}{\beta_1}  \le  \frac{d+1}{d-1},
\end{equation}
regardless of any property of the domain other than compactness.

In this connection, recall that
R. Ba\~{n}uelos and T. Kulczycki have proved in
\cite{BanKul2} that the fundamental gap of the Cauchy process is
controlled by the inradius in the case of a bounded convex domain
$\Omega$ of inradius $\rm{Inr(\Omega)}$,{\it viz.}, for $m=0$,
$$\beta_2-\beta_1\leq \frac{\sqrt{\lambda_2}-(1/2)\sqrt{\lambda_1}}{\rm Inr(\Omega)}.$$
where $\lambda_1$ and $\lambda_2$ are the first and second
eigenvalues for the Dirichlet Laplacian for the unit ball, $B_1$ in $\mathbb{R}^d$.
(Recall that the inradius ${\rm Inr(\Omega)}$ of a
region $\Omega$ is defined by
$${\rm Inr(\Omega)}=\sup\{d({\bf x}):  {\bf x} \in \Omega\},$$
where $d({\bf x})=\min\{|{\bf x}-{\bf y}|:{\bf y}\notin\Omega\}$ {\rm \cite{Dav}}.)

Since a ratio bound like \eqref{funrat} is algebraically equivalent to a gap
bound, \eqref{funrat} provides an independent upper bound on the
gap $\beta_2-\beta_1$. Continuing to set $m=0$, \eqref{ibbl}
and \eqref{funrat} in the form $\beta_2-\beta_1\leq
\frac{2}{d-1} \beta_1$ imply:

\begin{corollary}\label{corbl} If $\beta_1^*$ and $\lambda_1^*$ denote the fundamental
eigenvalues of $H_{0,\Omega}$ and $-\Delta$, respectively, on
the unit ball of ${\mathbb R}^d$, then
\begin{equation}\label{iwinr}
\beta_2-\beta_1\leq \left(\frac{2}{d-1}\right)\frac{\beta_1^*}{\rm
Inr(\Omega)}\leq
\left(\frac{2}{d-1}\right)\frac{\sqrt{\lambda_1^*}}{{\rm
Inr(\Omega)}}.
\end{equation}
\end{corollary}

\begin{proofof}{Corollary~\ref{corbl}} Since $H_{0,\Omega}$ is defined by closure from
a core of functions in $C_c^{\infty}$, its fundamental eigenvalue
satisfies the principle of domain monotonicity. That is, if
$\Omega_1\supset\Omega_2$, then
$\beta_1(\Omega_1)\leq\beta_1(\Omega_2)$. In particular, if $\Omega$
is a ball of radius $r$, then
$\displaystyle{\beta_1(\Omega)\leq\frac{\beta_1^*}{r}}$, which is
the fundamental eigenvalue of
the unit ball
$B_1$ by scaling. The first inequality
follows from \eqref{funrat}, and the second one by \eqref{ibbl}
\end{proofof}

\begin{proofof}{Theorem~\ref{upperbd}}
We make the special choice $H = H_{m,\Omega}$ and calculate
the first and second commutators with the aid of the Fourier
transform:

Writing $H_{m,\Omega}=\chi_\Omega {\mathcal F}^{-1}\sqrt{|{\bf
\xi}|^2+m^2}{\mathcal F}$,
\begin{eqnarray}
\left[H_{m,\Omega},x_{\alpha}\right]\varphi
&=& (H_{m,\Omega}\;x_{\alpha}-x_{\alpha}H_{m,\Omega})\varphi\nonumber\\
&=& \chi_\Omega{\mathcal F}^{-1}\sqrt{|{\bf \xi}|^2+m^2}{\mathcal
F}[x_{\alpha}\varphi]-\chi_\Omega x_{\alpha}
{\mathcal F}^{-1}[\sqrt{|{\bf \xi}|^2+m^2}\hat{\varphi}]\nonumber\\
&=&
\chi_\Omega {\mathcal F}^{-1} \left[\sqrt{|{\bf \xi}|^2+m^2}\frac{\partial\hat{\varphi}}{\partial
\xi_{\alpha}}-\frac{\partial}{\partial
\xi_{\alpha}}(\sqrt{|{\bf \xi}|^2+m^2}\hat\varphi)\right]\nonumber\\
&=& -i \chi_\Omega {\mathcal F}^{-1} \frac{\xi_{\alpha}}{\sqrt{|{\bf \xi}|^2+m^2}}
\hat{\varphi}.\label{fcomm}
\end{eqnarray}
Similarly,
\begin{equation}[x_{\alpha},[H_{m,\Omega},x_{\alpha}]]\varphi
=\chi_\Omega {\mathcal F}^{-1}\left[\left(\frac{1}{\sqrt{|{\bf
\xi}|^2+m^2}}- \frac{{\xi_{\alpha}}^2}{(|{\bf
\xi}|^2+m^2)^{3/2}}\right)
\hat{\varphi}\right].\label{scomm}\end{equation} Due to
\eqref{fcomm} and \eqref{scomm}, there are simplifications when we
sum over $\alpha$:
$$
\sum_{\alpha=1}^{d}{\|\left[H_{m,\Omega},x_{\alpha}\right]\varphi\|^2}
\leq \left\langle{\hat{\varphi}, \frac{|{\bf \xi}|^2}{|{\bf
\xi}|^2+m^2} \hat{\varphi}}\right\rangle \le 1,
$$
and
$$
\sum_{\alpha=1}^{d}{\left(\frac{1}{\sqrt{|{\bf \xi}|^2+m^2}}-
\frac{{\xi_{\alpha}}^2}{(|{\bf \xi}|^2+m^2)^{3/2}}\right)}
=\frac{(d-1)|{\bf \xi}|^2+ d \, m^2 }{(|{\bf \xi}|^2+m^2)^{3/2}}
$$
$$\quad\quad\quad\quad\quad\quad\quad\quad\quad\quad\quad\quad\quad
\ge \frac{d-1}{\sqrt{|{\bf \xi}|^2+m^2}}.
$$
In consequence, \eqref{R2mod} implies that
\begin{equation}\label{UvsR1}
(d-1)\sum_{j=1}^n(z-\beta_j)^2\langle u_j,
H_{m,\Omega}^{-1}u_j\rangle-2\sum_{j}(z-\beta_j) \le 0,
\end{equation}
provided $z \in [\beta_n, \beta_{n+1}]$.
Because
$$
H_{m,\Omega}^{-1}u_j = \frac{1}{\beta_j} u_j,
$$
and
$$
(z-\beta_j) = -\frac{(z-\beta_j)(z-\beta_j - z)}{\beta_j},
$$
Eq. \eqref{UvsR1} can be rewritten as
\begin{equation}\label{moddiffineq}
(d+1)\sum_{j=1}^n\frac{(z-\beta_j)^2}{\beta_j}-2z\sum_{j=1}^n
\frac{(z-\beta_j)}{\beta_j}\leq 0,
\end{equation}
or, equivalently,
\begin{equation}\label{poly}
(d-1) \overline{\beta_n^{-1}} z^2 - 2 d z + (d+1) \overline{\beta_n} \leq 0.
\end{equation}
Setting $z=\beta_{n+1}$, we see that $\beta_{n+1}$ must be less
than the larger root of \eqref{poly}, which is the conclusion of the theorem.
\end{proofof}

For future purposes we note that this theorem extends with small modifications
to semirelativistic Hamiltonians of the form $H_{m,\Omega} + V({\bf x})$.
More specifically,
\eqref{UvsR1} is valid when $\{u_k\}$ and $\{\beta_k\}$ are the
eigenfunctions and eigenvalues
of $H_{m,\Omega} + V({\bf x})$.

We next apply similar reasoning to a function related to Riesz
means. With $a_+ := max(0, a)$, let
\begin{equation}\label{Udef}
U(z):=\sum_{k}\frac{(z-\beta_k)_+^2}{\beta_k},
\end{equation}
where $z$ is a real variable. Note that if $z \in
[\beta_j,\beta_{j+1}]$, then
\begin{equation}\label{Uexpl}
\frac{U(z)}{j} = \overline{\beta_j^{-1}} z^2 -   2 z +
\overline{\beta_j}.
\end{equation}

\begin{theorem}\label{uindf} The function
$\displaystyle{z^{-(d+1)}U(z)}$ is nondecreasing in the variable
$z$. Moreover, for $d \ge 2$ and any $j \ge 1$, the  ``Riesz mean''
$R_1(z):=\sum_{k}(z-\beta_k)_+$ satisfies
\begin{equation}\label{r1}
R_1(z)\geq
\left(\frac{2j(d-1)^d}{(d+1)^{d+1}\overline{\beta_j}^{\;d}}
\right)z^{d+1}
\end{equation}
for all $\displaystyle{z\geq
\left(\frac{d+1}{d-1}\right)\overline{\beta_j}}$.
\end{theorem}

\begin{proof}
In notation that suppresses $n$, Eq. \eqref{moddiffineq} can be written
\begin{equation}
(d+1)\sum_{k}\frac{(z-\beta_k)_+^2}{\beta_k}-2z\sum_{k}
\frac{(z-\beta_k)_+}{\beta_k}\leq 0,
\end{equation}
which for the function $U$ reads
\begin{equation} (d+1)U(z)-zU^{'}(z) \leq 0,\\ \nonumber
\end{equation}
or, equivalently,
\begin{equation}
\frac{d}{dz}\left\{\frac{U(z)}{z^{d+1}} \right\} \geq 0,
\end{equation}
proving the claim about $U$.

Eq. \eqref{UvsR1} tells us that
\begin{equation}\label{ibru}
R_1(z)\geq \frac{d-1}{2}\; U(z).
\end{equation}

Since  $\displaystyle{\frac{U(z)}{z^{d+1}}}$ is nondecreasing, when
$z\geq z_{j*}\geq\beta_j$,
\begin{equation}\label{unondec}
U(z)\geq\left( \frac{z}{z_{j*}}\right)^{d+1}U(z_{j*}).\end{equation}

From \eqref{Uexpl} with the Cauchy-Schwarz inequality we get
\begin{equation}\label{Ulower}
\frac{U(z)}{j} \ge \frac{1}{\overline{\beta_j}}(z -
\overline{\beta_j})^2,
\end{equation}
so that with \eqref{ibru} and \eqref{unondec}
we obtain
\begin{equation}\label{r1geqzp}
R_1(z)\geq
\frac{(d-1)j}{2\overline{\beta_j}}\left(\frac{z}{z_{j*}}\right)^{d+1}
\left(z_{j*} - \overline{\beta_j}\right)^2.
\end{equation}
We now choose an optimized value of $z_{j*}$ to maximize the
coefficient of $z^{d+1}$, {\it viz}.,
$\displaystyle{z_{j*}=\frac{d+1}{d-1}\overline{\beta_j}}$.
Substituting this into (\ref{r1geqzp}), we get \eqref{r1}, as
claimed.
\end{proof}

The Legendre transform of $R_1(z)$ is a straightforward calculation,
to be found explicitly for example in \cite{HarHer,LapWeidl}. The
result for $k-1<w<k$ is
\begin{equation}
R^*_1(w) =(w-[w])\beta_{[w]+1}+[w]\overline{\beta_{[w]}}\label{ltr},
\end{equation}
where $[w]$ denotes the greatest integer $\le w$, and when $w$ takes
an integer value $k$ from below, $R_1^*(k)=k\overline\beta_k$.

With the Legendre transform of the right side of (\ref{r1}),
we get
\begin{equation}\label{ltap}
k\overline{\beta}_{k} \leq
\frac{d~\overline{\beta}_j}{2^{1/d}\,j^{1/d}(d-1)}k^{\frac{d+1}{d}}.
\end{equation}
This leads us to the following upper bound for ratios of averages of
eigenvalues of $H_{m,\Omega}$:

\begin{corollary}  For $k>2j$, Eq. (\ref{ltap}) implies
\begin{equation}
\frac{\overline{\beta}_{k}}{\overline{\beta}_{j}} \leq
\frac{d}{2^{1/d}(d-1)}\left(\frac{k}{j}\right)^{\frac{1}{d}}\label{akp1b}.
\end{equation}
\end{corollary}
\begin{remark}
The reason for the restriction on $k,j$ is that in
Theorem \ref{uindf}, we assumed that  $\displaystyle{z\geq
\left(\frac{d+1}{d-1}\right)\overline{\beta_j}}$. Since there is a
monotonic relationship between $w$ and the maximizing value of $z_*$
in the calculation of the Legendre transform of the right side of
(\ref{r1}) , we get
\begin{equation}
w=2 j \left(\frac{(d-1) z_*}{
 \left(d+1\right)\overline{\beta_j}}\right)^d.
\end{equation}
Thus the inequality is valid under the assumption that
$\displaystyle{k> w \geq 2j}$.

\end{remark}

\section{Weyl asymptotics and semiclassical bounds for $H_{m,\Omega}$}

In this section we consider the eigenvalues $\beta_k$ of
$H_{m,\Omega}$ as $k \rightarrow \infty$. In view of the elementary
inequalities \eqref{H0vsHm}, and the fact that
$\displaystyle{\lim_{|\xi|\to
\infty}\frac{\sqrt{|\xi|^2+m^2}}{|\xi|}=1}$, it suffices to consider
the case $m=0$.\\

\noindent We begin with the analogue of the Weyl formula for the
Laplacian, adapting one of the standard proofs of the latter, which
relies on an estimate of the partition function
$\displaystyle{Z(t):=\sum e^{-\beta_j t}}$ for $t>0$. Recall that
the function $Z(t)$ can be written as
\begin{equation}\label{ZN}
Z(t)=\int e^{-\beta t}dN(\beta),
\end{equation}
where $\displaystyle{N(\beta):=\sum_{\beta_j\leq \beta} 1}~$ is the
usual counting function.  Another standard formula for the partition
function is
\begin{equation}\label{Ztrace}
Z(t) = \int_{\Omega}{p_{\Omega}({\mathbf x},{\mathbf x},t) d{\mathbf
x}}.
\end{equation}
If we accept that $H_{m,\Omega}$ is well approximated by
$\sqrt{-\Delta_{\Omega}}$ in the ``semiclassical limit,'' then the analogue
for $N(\beta)$ of the Weyl asymptotic formula for the Laplacian
should be identical to the usual Weyl formula, with the
identification of $\beta_k$ with $\sqrt{\lambda_k}$. This intuition
is confirmed by the following:

\begin{proposition}{\label{asym}} As $\beta \rightarrow \infty$,
\begin{equation}\label{WeylN}
N(\beta)\sim   \frac{|\Omega|}{(4 \pi)^{d/2}\Gamma(1+d/2)} \beta^d.
\end{equation}
\end{proposition}
Equivalently, as $k\to \infty$,
\begin{equation}\label{wlfb}
\beta_k \sim  \sqrt{4 \pi} \left(\frac{\Gamma(1+d/2) k}{|\Omega|}\right)^{1/d}.
\end{equation}
Moreover, the function $U$ of \eqref{Udef} satisfies
\begin{equation}\nonumber
U(z) \sim \frac{|\Omega|}{2\pi^{d/2} (d^2-1) \Gamma(1+d/2)}z^{d+1}.
\end{equation}

\begin{proof} By Karamata's Tauberian theorem \cite{Sim}, if we can show that
for $t \rightarrow 0$,
$$t^dZ(t)\to c_d|\Omega|,$$
then the first claim follows from \eqref{ZN}.
The further claims for $\beta_k$ and $U(z)$ are easy consequences of
\eqref{WeylN}.

By a standard comparison,
\begin{equation}\label{inpomegap0}p_{\Omega}({\mathbf x},{\mathbf y},t) < p_0({\mathbf
x}-{\mathbf y},t)\end{equation} on $\Omega$, where $p_{\Omega}$ is
the integral kernel of the semigroup $\displaystyle{e^{-tH_{0,\Omega}}}$. Define
$$r_{\Omega}:=p_0({\mathbf x}-{\mathbf y},t)-p_{\Omega}({\mathbf x},{\mathbf y},t),$$
and let $\delta_{\Omega}({\mathbf x}):={\rm dist}({\mathbf
x},\partial \Omega)$.  According to \cite{BanKul},
$$0\leq r_{\Omega}\leq \frac{t}{\delta_{\Omega}^{d+1}(x)}c_d {\mathcal P}^{{\mathbf y}}(\tau_{\Omega}<t),$$
where ${\mathcal P}^{\mathbf y}(\tau_{\Omega})$ is the probability
that a path originating at ${\mathbf y}$ exits $\Omega$ before time $t$. Thus,
$$\int_{\Omega} p_{\Omega}({\mathbf x},{\mathbf x},t)
\leq \int_{\Omega} p_0({\mathbf 0},t)=c_d\frac{|\Omega|}{t^d}.$$ It
is shown in \cite{BanKul} that ${\mathcal P}^{{\mathbf
y}}(\tau_{\Omega}<t)\to 0$ as $t\to 0^+$, and we proceed to
calculate:
\begin{eqnarray} \nonumber \int_{\Omega}
p_{\Omega}({\mathbf x},{\mathbf x},t)d{\mathbf x} &=& \int_{\Omega}
p_0({\mathbf 0},t)d{\mathbf x}-\int_{\Omega} r_{\Omega}d{\mathbf x}\\
\nonumber&\geq& c_d\frac{|\Omega|}{t^d}- (o(1_t))\cdot
\left\{\int_{\{{\mathbf x}:~ \delta({\mathbf
x})<\sqrt{t}\}}\frac{t}{\delta_{\Omega}^{d+1}({\mathbf x})}+
\int_{\{{\mathbf x}:~
\delta({\mathbf x})>\sqrt{t}\}}\frac{t}{\delta_{\Omega}^{d+1}({\mathbf x})}\right\}.\\
&\quad& \label{twoint}
\end{eqnarray}
The first integral on the right side of (\ref{twoint}) becomes
\begin{eqnarray}\int_{\{{\mathbf x}: \delta({\mathbf
x})<\sqrt{t}\}}r_{\Omega}d{\mathbf x}&=&
\int_{\{{\mathbf
x}:~ \delta({\mathbf
x})<\sqrt{t}\}}\frac{t}{\delta_{\Omega}^{d+1}({\mathbf x})}d{\mathbf x} \nonumber\\
&\leq&
C\int_{\Omega-\Omega_{\sqrt{t}}}\frac{t}{(t^2)^{(d+1)/2}}d{\mathbf
x}\nonumber\\
&=&C \, t^{-d}|\Omega-\Omega_{\sqrt{t}}|. \label{int1}
\end{eqnarray}
As for the second integral,
\begin{eqnarray}\int_{\{{\mathbf x}:
\delta({\mathbf x})>\sqrt{t}\}}r_{\Omega}d{\mathbf
x}&=&\int_{\{{\mathbf x}:~ \delta({\mathbf
x})>\sqrt{t}\}}\frac{t}{\delta_{\Omega}^{d+1}({\mathbf x})}d{\mathbf
x}\nonumber\\
&\leq& \frac{t}{(t)^{(d+1)/2}}|\Omega|\nonumber\\
&=& 0(t^{(1-d)/2})<<t^{-d}. \label{int2}\end{eqnarray}
With
(\ref{int1}) we thus validate the condition allowing the application of
Karamata's Tauberian Theorem.
\end{proof}

An easy corollary of Theorem \ref{uindf} is a counterpart for $H_{0,\Omega}$ to the
Li-Yau inequality for the
Laplacian \cite{LiYau}.  (As noted in \cite{LapWeidl}, the Li-Yau inequality
is equivalent to an earlier inequality by Berezin \cite{Bere} through the Legendre transform.
See also \cite{LieLos}.)

Since we know that
$z^{-(d+1)} U(z) \uparrow \frac{2c_d|\Omega|}{d!(d^2-1)}$,
and that because of \eqref{simpleupper}, a choice of $z$
safely guaranteed to exceed $\beta_k$ is $z = \frac{d+1}{d-1}
\overline{\beta_k}$, with the aid of \eqref{Ulower} we obtain
$$
\frac{2c_d|\Omega|}{d!(d^2-1)} \ge
 \frac{k}{\overline{\beta_k} }  \left( \frac{2}{d-1}\overline{\beta_k}  \right)^2 \left(\frac{d+1}{d-1} \overline{\beta_k}\right)^{-(d+1)}.
$$
This leads directly to the semiclassical estimate:
\begin{equation}\label{OldGenLiYau}
{\overline{\beta_k}} \ge \frac{(d-1) 2^{1/d} \sqrt{4 \pi}}{d+1} \left(\frac{\Gamma(1+d/2) k}{|\Omega|}\right)^{1/d}.
\end{equation}
However, a better estimate, improving $(d-1)2^{1/d}$ to $d$,
can be derived by following the argument of
Li and Yau \cite{LiYau} more closely.
As a first step we slightly generalize the lemma attributed in \cite{LiYau} to
H\"ormander:

\begin{lemma}\label{hormander}
Let $f: {\mathbb R}^d \to {\mathbb R}$ satisfy $0\le f({\bf \xi})\le
M_1$ and
\begin{equation}
\int_{{\mathbb R}^d} f({\bf \xi})w(|{\bf \xi}|)d{\bf \xi}\le M_2,\end{equation}
where the weight function $w$ is nonnegative and nondecreasing.
Define $R=R(M_1,M_2)$ by the condition that
\begin{equation}
\int_{B_R} w(|{\bf \xi}|)d{\bf \xi}=\omega_{d-1}\int^R_0 w(r) r^{d-1} dr=
\frac{M_2}{M_1},
\end{equation}
where $\omega_{d-1} := |{\mathbb S}^{d-1}| = \frac{2 \pi^{d/2}}{\Gamma(d/2)}$.  Then
\begin{equation}
\int_{{\mathbb R}^d} f({\bf \xi})d{\bf \xi}\le \frac{ \pi^{d/2} M_1}{\Gamma(1+d/2)} R^d.
\end{equation}
As a special case, if $w({\bf \xi})=|{\bf \xi}|^p$, then
$R=\left[\frac{M_2(d+p)}{M_1w_{d-1}}\right]^{\frac1{d+p}}$, and so
\begin{align*}
\int_{{\mathbb R}^d}f({\bf \xi})d{\bf \xi}&\le \frac1d ((d+p)M_2)^{\frac d{d+p}}
(w_{d-1}M_1)^{\frac {-p}{d+p}}\\
&=\left(\frac{d+p}{d} M_2\right)^{\frac d{d+p}}
\left(\frac{\pi^{d/2} M_1}{\Gamma(1+d/2)}\right)^{\frac
p{d+p}}.\end{align*}
\end{lemma}

{\bf Proof.}  Let $g({\bf \xi}):=M_1\chi_{\{|{\bf \xi}|\le R\}}$ and note
that according to the definition of $R$, $\int w(|{\bf \xi}|)g({\bf \xi})d{\bf \xi}=M_2$.
We observe that
$(w(|{\bf \xi}|)-w(R))(f({\bf \xi})-g({\bf \xi}))\ge 0$ for all
${\bf \xi}$.  (Check $|{\bf \xi}|\le R$ and $|{\bf \xi}|>R$ separately.)
Hence
\begin{equation}
w(R)\int (f({\bf \xi})-g({\bf \xi}))d{\bf \xi}\le\int w(|{\bf \xi}|)
(f({\bf \xi})-g({\bf \xi}))=0,
\end{equation}
and, consequently,
\begin{equation}
\int f({\bf \xi})d{\bf \xi}\le \int g({\bf \xi})d{\bf \xi}=
|B_R| M_1=\frac{\pi^{d/2} M_1}{\Gamma(1+d/2)} R^d.
\end{equation}
\hfill$\Box$

For the application to $H_{0,\Omega}$, note that
\begin{equation}
\beta_\ell=\langle u_\ell, H_{0,\Omega} u_\ell\rangle=\int |{\bf \xi}|
|\hat u_\ell ({\bf \xi})|^2 d{\bf \xi}
\end{equation}
Choosing $w(|{\bf \xi}|)=|{\bf \xi}|$ in the lemma, with
$f({\bf \xi})=\sum^k_{\ell=1} |\hat u_\ell ({\bf \xi})|^2$, we find
\begin{equation}
k=\int f({\bf \xi})d{\bf \xi} \le \left(\|f\|_\infty \frac{\pi^{d/2}}{\Gamma(1+d/2)}\right)^{\frac 1{d+1}}
\left(\left(\sum^k_{\ell=1}\beta_\ell\right)\frac{d+1}d\right)^{\frac d{d+1}},
\end{equation}
or
\begin{equation}
\sum^k_{\ell=1}\beta_\ell\ge \frac d{d+1}\left(\frac{\Gamma(1+d/2)}{\pi^{d/2}\|f\|_\infty}
\right)^{1/d} k^{1+\frac1d}.\end{equation}
As for $\|f\|_\infty$,
\begin{align*}
\sum^k_{\ell=1} |\hat u_\ell ({\bf \xi})|^2&=\sum^k_{\ell=1} \frac1{(2\pi)^d}
\Bigl|\int_\Omega e^{i{\bf x}\cdot{\bf \xi}}u_\ell ({\bf x}) d{\bf x}\Bigr|^2\\
&=\frac1{(2\pi)^d}\sum^k_{\ell=1} \left| \langle e^{i{\bf x}\cdot{\bf \xi}},
u_k\rangle\right|^2\le \frac{|\Omega|}{(2\pi)^d}
\end{align*}
by Bessel's inequality, as $\|e^{i{\bf x}\cdot{\bf \xi}}\|^{2}_2=
|\Omega|$.
In conclusion, we have an analogue of the Li-Yau inequality \cite{LiYau}:

\begin{theorem}\label{GLY}
For all $k=1,\dots$, the eigenvalues
$\beta_k$ of $|{\mathbf P}|_{\Omega}$
satisfy
\begin{equation}\label{GenLiYau}
\overline{\beta_k}
\ge \frac{ \sqrt{4 \pi} d}{d+1}\left(\frac{\Gamma(1+d/2)k}{|\Omega|}\right)^{1/d}.
\end{equation}
\end{theorem}

We observe that, just like the Li-Yau inequality for the Laplacian,
\eqref{GenLiYau} has the best possible coefficient consistent with
the Weyl-type law of Proposition \ref{asym}. Moreover, in view of
\eqref{ibbl}, Theorem \ref{GLY} has a corollary for the Dirichlet
Laplacian:
\begin{equation}\label{SqrtLiYau}
\frac{1}{k}\sum_{\ell=1}^k{\sqrt{\lambda_\ell}}
\ge \frac{ \sqrt{4 \pi} d}{d+1}\left(\frac{\Gamma(1+d/2)k}{|\Omega|}\right)^{1/d},
\end{equation}
which is comparable to the Li-Yau inequality, but neither implies it
nor is directly implied by it .  (For an alternative route to \eqref{SqrtLiYau}
see Theorem 5.1 of \cite{HarStu2}.)

\section{Universal bounds for $H_{m,\Omega} + V({\bf x})$}\label{quantum}

We turn now to the Klein-Gordon Hamiltonian with an external interaction,
\begin{equation}\label{Ham2}
H = H_{m,\Omega} + V({\bf x}).
\end{equation}
In a semi-relativistic approximation this Hamiltonian models the
motion of a spinless particle in an external force field. As
mentioned above, Hamiltonian operators similar to \eqref{Ham2} have
also recently been of interest as models of nonrelativistic charge
carriers traveling in a two-dimensional hexagonal structure like
carbon graphene. (What distinguishes graphene from the common
material graphite is that graphene sheets are only one atom thick.)
This material has been the subject of intense study recently because
of its remarkable electronic properties.  Due to the special
symmetry of the hexagonal lattice, standard approximations in
condensed-matter theory do not lead to the usual effective mass
approximation for charge carriers, but rather, they behave like
relativistic particles with a reduced ``speed of light.'' On the
theoretical side this has been known since 1947 when the
(unintegrated) density of states at low energies was calculated in a
tight-binding approximation and found to be proportional to $|E -
E_0|$ as a function of energy $E$, as is the case for a
two-dimensional relativistic particle \cite{Wal}. Confirming
experiments date from the past decade (e.g.,
\cite{NGMJKGDF,SadMarEtc,deHBerEtc}), where the charge carriers are
electrons. A calculation of the density of states does not in fact
allow an unambiguous determination of the effective Hamiltonian of
particles moving in graphene, so the details of models used in the
physical literature vary. Furthermore, although the standard
effective-mass approximation for periodic Schr\"odinger Hamiltonians
has had a rigorous mathematical basis since the work of Odeh and
Keller \cite{OdKe} (see also \cite{Bus,GuRaTr,Nen,GeMaSj}), we are
unaware of comparably convincing analysis of the effective
Hamiltonian for materials like graphene that offer a clear
prescription for treating boundaries. The practice in the physical
literature has been to propose relativistic Hamiltonians with {\it
ad hoc} modifications to account for the effect of the boundary
geometry, the effects of have become accessible to experiment quite
recently (e.g., \cite{MogZar,NFDD,PSKY,BadKor}). For a sampling of
the different graphene-related models and calculations, see
\cite{Wal,Spa,Sem,JacPi,Ran}. Because the usual charge carrier is an
electron, which is a spin $\frac1 2$ particle, more often than not
the Hamiltonian is chosen as a Dirac operator acting on the set of
two-component spinors. We hope to elaborate the spectral theory of
Hamiltonians with spin in future work, but in the present work we
content ourselves with the study of \eqref{Ham2}, and we also
continue to restrict the Hamiltonian to a finite domain in order to
achieve a discrete spectrum.  Our point of departure to derive
useful spectral bounds for \eqref{Ham2} is \eqref{UvsR1}, which
remains valid for interacting operators $H$.
\begin{theorem}
Let $\beta_k$ denote the eigenvalues of \eqref{Ham2}, and set

$\displaystyle{U(z):=\sum_{k}\frac{(z-\beta_k)_+^2}{\beta_k}}$ as in
\eqref{Udef}.  Assume that the measurable function

$V=V_+-V_-$ with
$V_{\pm}\geq 0$ and $V_-\in L^s$ for some $\displaystyle{2\leq
d<s<\infty}$. If
\begin{equation}\label{V-est}\|V_-\|_s <
\frac{\sqrt{\pi}2^{\frac{(d-1)^2}{d}}\Gamma\left(\frac{d}{2}\right)^{\frac{1-2d}{d}}
(d|\Omega|)^{\frac{d-s}{sd}}(s-d)^{\frac{s-1}{s}}}
{(d-2)!(s-1)^{\frac{s-1}{s}}},\end{equation} let us define $\alpha <
1$ by
$$\alpha:=\frac{\|V_-\|_s(d-2)!(s-1)^{\frac{s-1}{s}}}{\sqrt{\pi}2^{\frac{(d-1)^2}{d}}
\Gamma\left(\frac{d}{2}\right)^{\frac{1-2d}{d}}
(d|\Omega|)^{\frac{d-s}{sd}}(s-d)^{\frac{s-1}{s}}}.$$
Then for each
$k$, the eigenvalues $\beta_k$ satisfy
\begin{equation}\label{funrat2}
\frac{\beta_{k+1}}{\overline{\beta_k}}
\le \overline{\beta_k^{-1}} {\beta_{k+1}}   \le 1+\frac{2}{(d-1)(1-\alpha)}.
\end{equation}
Moreover, $\displaystyle{\frac{U(z)}{z^{((d+1)-\alpha(d-1))}}}$ is a
nondecreasing function of $z\in{\mathbb R}$, and for $k > 2j$,
\begin{equation}\label{modifratio}\frac{\overline{{\beta_k}}}{\overline{\beta_j}}\leq
\frac{d-\alpha(d-1)}{(d-1)(1-\alpha)2^{1/(d-\alpha(d-1))}}\left(
\frac{k}{j}\right)^{1/(d-\alpha(d-1))}.
\end{equation}
\end{theorem}
\begin{proof}
From \eqref{UvsR1},
\begin{equation}\label{UvsR1again}
(d-1)\sum_{j=1}^n(z-\beta_j)^2\langle u_j,
H_{m,\Omega}^{-1}u_j\rangle-2\sum_{j}(z-\beta_j) \le 0.
\end{equation}
Since $V_{\pm}\geq 0$,
$$H_{m,\Omega}+V > H_{m,\Omega}-V_-,$$
and so
$$(H_{m,\Omega}+V)^{-1}{\le}(H_{m,\Omega}-V_-)^{-1}.$$
Hence,
\begin{eqnarray*}
\frac{1}{\beta_j} = \langle u_j,(H_{m,\Omega}+V)^{-1}u_j\rangle
&{\le}&\langle
u_j,(H_{m,\Omega}-V_-)^{-1}u_j\rangle\\
&{\le}& \langle
u_j,(H_{m,\Omega}-V_-)^{-1}u_j\rangle\\
&=&\langle u_j,H_{m,\Omega}^{-1}u_j\rangle+
\langle u_j,(H_{m,\Omega}-V_-)^{-1}V_-H_{m,\Omega}^{-1}u_j\rangle, \\
\end{eqnarray*}
according to the resolvent formula.  Since
$$\langle u_j,(H_{m,\Omega}-V_-)^{-1}V_-H_{m,\Omega}^{-1}u_j\rangle
=\frac{1}{\beta_j}\langle
u_j,V_-H_{m,\Omega}^{-1}u_j\rangle,
$$
\begin{equation}\label{1-hinv}
\frac{1}{\beta_j}\left( 1-\langle u_j,V_-H_{m,\Omega}^{-1}u_j\rangle
\right){\le} \langle u_j,H_{m,\Omega}^{-1}u_j\rangle.\end{equation}
If $2\leq d\leq s<\infty$, we now claim that
\begin{equation}\label{vhinv}\|V_-H_{m,\Omega}^{-1}\varphi\|_2\leq
\alpha\|\varphi\|_2\end{equation} for any $\varphi\in L^2$. Granting
the claim, with  $\varphi=u_j$ in \eqref{1-hinv}, we get
\begin{equation} \label{vhinv2} \frac{1-\alpha}{\beta_j}{\le}\langle
u_j,H_{m,\Omega}^{-1}u_j\rangle.\end{equation} To establish
\eqref{vhinv} begin by noting that by H\"{o}lder's inequality,
\begin{equation}\label{eq1}
\|V_-H_{m,\Omega}^{-1}\varphi\|_2\leq
\|V_-\|_s\|H_{m,\Omega}^{-1}\varphi\|_{\frac{2s}{s-2}}.
\end{equation}
Because $H_{m,\Omega}\ge H_{0,\Omega}$,
\begin{equation}\label{eq2}
\|H_{m,\Omega}^{-1}\varphi\|_{\frac{2s}{s-2}}{\le}
 \|H_{0,\Omega}^{-1}\varphi\|_{\frac{2s}{s-2}}.
 \end{equation}
Inequality \eqref{inpomegap0} for the transition density implies
\begin{equation*}e^{-tH_{0,\Omega}}({\bf x},{\bf y},t)\leq
{p_0({\bf x}-{\bf y},t)} = \frac{-c_d}{d-1}\frac{\partial}{\partial
t} \left(t^2+|{\bf x}-{\bf
y}|^2\right)^{-\left(\frac{d-1}{2}\right)}.\end{equation*}
Applying the Laplace transform, the kernel of $H_{0,\Omega}^{-1}$ is
less than
$$ \int_0^{\infty}{\left(\frac{-c_d}{d-1}\frac{\partial}{\partial t}
\left(t^2+|{\bf x}-{\bf
y}|^2\right)^{-\left(\frac{d-1}{2}\right)}\right) dt} =
\frac{c_d}{d-1} |{\bf x}-{\bf y}|^{-(d-1)}.$$
Together with
\eqref{eq1} and \eqref{eq2} we get
\begin{equation*}\label{eq4}\|V_-H_{m,\Omega}^{-1}\varphi\|_2\leq
\frac{c_d}{d-1}\|V_-\|_s\||{\bf x}|^{-(d-1)}*\varphi\|_{\frac{2s}{s-2}}.
\end{equation*}
According to Young's convolution inequality,
$$\||{\bf x}|^{-(d-1)}*\varphi\|_{\frac{2s}{s-2}}\leq
\||{\bf x}|^{-(d-1)}\|_{\frac{s}{s-1}}\|\varphi\|_2,$$
so
\begin{equation}\label{vpinv}
\left\|
V_-H_{m,\Omega}^{-1}\varphi\right\|_2 \leq
\frac{\Gamma\left(\frac{d+1}{2}\right)}{\pi^{(d+1)/2}
(d-1)}
\|V_-\|_{s}\left\||{\bf x}|^{-(d-1)}\right\|_{\frac{s}{s-1}}
\|\varphi\|_2.
\end{equation}
For an upper bound to $\left\||{\bf x}|^{-(d-1)}\right\|_{\frac{s}{s-1}}$,
choose $R^*$ as the radius of the ball $B_{R^*}$ centered at the
origin having the same volume as $\Omega$. Since by rearrangement,
$$\||{\bf x}|^{-(d-1)}\|_{L^{\frac{s}{s-1}}(\Omega) }\leq
\||{\bf x}|^{-(d-1)}\|_{L^{\frac{s}{s-1}}(B_{R^*}) } =
\left(\omega_{d-1}\frac{(R^*)^{\frac{s-d}{s-1}}(s-1)}{s-d}
\right)^{\frac{s-1}{s}},$$ we get the estimate
\begin{equation}\label{nof|x|}\||{\bf x}|^{-(d-1)}\|_{\frac{s}{s-1}}<
2^{\frac{d-1}{d}}\pi^{\frac{d-1}{2}}\left[\Gamma\left(\frac{d}{2}\right)\right]^{\frac{1-d}{d}}
{(d|\Omega|)^{\frac{s-d}{sd}}}{\left(\frac{s-1}{s-d}\right)^{\frac{s-1}{s}}}.
\end{equation}
With
\eqref{vpinv} and \eqref{V-est} this implies \eqref{vhinv}
and consequently
\eqref{vhinv2}.
Because $\alpha < 1$ by assumption, \eqref{UvsR1again} together with
\eqref{vhinv2} yield
\begin{equation}\label{newineq}
(d-1)\sum_{j=1}^{n}\frac{1-\alpha}{\beta_j}(z-\beta_j)^2
-2\sum_{j=1}^{n}(z-\beta_j)\leq 0,
\end{equation}
or, equivalently,
\begin{equation}\label{newequiv}(d-1)(1-\alpha)\overline{\beta_k^{-1}}z^2-
2[d-\alpha(d-1)]z+[d+1-\alpha(d-1)]\overline{\beta_k}\le
0.\end{equation} By setting $z=\beta_{k+1}$, we see that
$\beta_{k+1}$ must be smaller than the larger root of
\eqref{newequiv}, i.e., after some algebra,
\begin{equation}\label{ugly2} \beta_{k+1} \le
\frac{(d-1)(1-\alpha) + 1 + \sqrt{1 - \left((d+\alpha - \alpha d)^2
-1\right) \left(\overline{\beta_k} \,\overline{\beta_k^{-1}} -
1\right)}} {(d-1)(1-\alpha)\overline{\beta_k^{-1}}}.
\end{equation}
\noindent
As was the case for \eqref{simpleupper}, with
the Cauchy-Schwarz inequality in the form $1\leq
\overline{\beta_k}\overline{\beta_k^{-1}}$, \eqref{ugly2} implies
the simpler but slightly weaker inequalities
\eqref{funrat2}.\\

Now observe that \eqref{newineq} differs from \eqref{moddiffineq}
only in the extra factor $1-\alpha{\, > 0}$, and therefore all of
the consequences of that inequality can be recovered with suitable
changes of some constants.  In particular, the function
$\displaystyle{\frac{U(z)}{z^{(d+1)-\alpha(d-1)}}}$ is
nondecreasing, and therefore,
\begin{equation} U(z)\geq
\left(\frac{z}{z_{j^*}}\right)^{(d+1)-\alpha(d-1)}U(z_{j^*})\end{equation}
when $z\geq z_{j^*}\geq \beta_j$. \\
At the same time, by
\eqref{newineq} we have
\begin{equation}\label{uvR1}
\frac{(d-1)(1-\alpha)}{2}U(z)\leq R_1(z).
\end{equation}
By \eqref{uvR1} and the fact that
$\displaystyle{\frac{U(z)}{j}\geq\frac{1}{\overline{\beta}_j}
(z-\overline{\beta}_j)^2},$ we obtain
\begin{equation}\label{newr1geqzp}
R_1(z)\geq
\frac{(d-1)(1-\alpha)j}{2\overline{\beta}_j}\left(\frac{z}{z_{j^*}}\right)^{(d+1)-\alpha(d-1)}
(z_{j^*}-\overline{\beta_j})^2.
\end{equation}
To maximize the coefficient of $\displaystyle{z^{d+1-\alpha(d-1)}}$
we optimize $z_{j^*}$ and get
$$\displaystyle{z_{j^*}=\frac{(d+1)-\alpha(d-1)}{(d-1)(1-\alpha)}\overline{\beta_j}}.$$
Substituting this into \eqref{newr1geqzp} gives
\begin{equation}\label{newr1}
R_1(z)\geq \frac{2j[(d-1)(1-\alpha)]^{d-\alpha(d-1)}}
{[(d+1)-\alpha(d-1)]^{(d+1)-\alpha(d-1)}\overline{\beta_j}^{~d-\alpha(d-1)}}z^{(d+1)-\alpha(d-1)}
\end{equation}
for all $\displaystyle{z\geq
\frac{(d+1)-\alpha(d-1)}{(d-1)(1-\alpha)}\overline{\beta_j}}$.

With the Legendre transform of the right hand side of \eqref{newr1},
we obtain
\begin{equation}k\overline{\beta_k}\leq \frac{[d-\alpha(d-1)]\overline{\beta_j}}
{[(d-1)(1-\alpha)]2^{1/(d-\alpha(d-1))}j^{1/(d-\alpha(d-1))}}k^{1+1/(d-\alpha(d-1))}.
\end{equation}
Therefore,
\begin{equation}
\frac{\overline{{\beta_k}}}{\overline{\beta_j}}\leq
\frac{d-\alpha(d-1)}{[(d-1)(1-\alpha)]2^{1/(d-\alpha(d-1))}}\left(
\frac{k}{j}\right)^{1/(d-\alpha(d-1))}
\end{equation}
as claimed.
\end{proof}

\bigskip
\subsection*{Acknowledgements}
The authors are grateful to Mark Ashbaugh, Lotfi Hermi, and Joachim
Stubbe for conversations and references.


\end{document}